\documentclass[onefignum,onetabnum]{siamart171218}

\usepackage{amsmath}
\usepackage{amsmath,amscd}
\usepackage{mathrsfs}
\usepackage{amssymb}
\usepackage{lipsum}
\usepackage{amsfonts}
\usepackage{graphicx}
\usepackage{epstopdf}
\usepackage{algorithmic}
\usepackage{multicol}
\newcommand{\suchthat}{\;\ifnum\currentgrouptype=16 \middle\fi|\;}
\usepackage{enumitem}

\newcommand{\scirc}{\raise1pt\hbox{$\,\scriptstyle\circ\,$}}
\newcommand{\R}{\mathbb{R}}

\newcommand{\N}{\mathbb{N}}
\newcommand{\Z}{\mathbb{Z}}

\newtheorem{assumption}{Assumption}

\DeclareSymbolFont{bbold}{U}{bbold}{m}{n}
\DeclareSymbolFontAlphabet{\mathbbold}{bbold}
\newcommand{\vect}[1]{\mathbbold{#1}}

\newsiamthm{conjecture}{Conjecture}
\newsiamthm{claim}{Claim}
\newsiamremark{rem}{Remark}
\newsiamremark{expl}{Example}
\newsiamremark{hypothesis}{Hypothesis}
\crefname{hypothesis}{Hypothesis}{Hypotheses}
\ifpdf
  \DeclareGraphicsExtensions{.eps,.pdf,.png,.jpg}
\else
  \DeclareGraphicsExtensions{.eps}
\fi

\newcommand{\TheTitle}{On small-time local controllability} 
\newcommand{\TheAuthors}{Saber Jafarpour}

\headers{\TheTitle}{\TheAuthors}

\title{{\TheTitle}}

\author{
  Saber Jafarpour\thanks{Department of Mechanical Engineering and
    Center for Control, Dynamical Systems, and Computations at 
    University of California, Santa Barbara
    (\email{saber.jafarpour@engineering.ucsb.edu}).}}

\usepackage{amsopn}

\ifpdf
\hypersetup{
  pdftitle={\TheTitle},
  pdfauthor={\TheAuthors}
}
\fi

\begin{document}

\maketitle

\begin{abstract}
In this paper, we study small-time local controllability of real
analytic control-affine systems under small perturbations of their vector fields. Consider a real analytic control
system $\mathcal{X}$ which is small-time locally controllable and whose reachable sets shrink with the polynomial rate of
order $N$ with respect to time. We will prove a general theorem which
states that any real analytic control-affine system whose vector fields are perturbations of the vector fields
of $\mathcal{X}$ with polynomials of order higher than $N$ is again
small-time locally controllable. In particular, we show that this result
connects two long-standing open conjectures about small-time local
controllability of systems. 
\end{abstract}



\section{Introduction}Controllability is one of the central concepts
in mathematical control theory. For linear control systems, 
the notion of controllability has been first introduced and studied by
Kalman~\cite{kalman1960contributions}. Based on the state-space
approach, Kalman characterized controllability using what is now known as the Kalman rank condition~\cite{kalman1960contributions,REK:63}. For nonlinear
systems, various notions of controllability have been introduced and studied in
the literature~\cite{HJS:1983}. Among these notions, small-time local controllability is arguably the most fundamental
one. A control system is small-time locally controllable from a point
if a neighborhood of that point can be reached in small times (a
rigorous definition will be given later in the paper). If one can
compute all the trajectories of the system, then it is easy to study the small-time local
controllability. However, a
full analytic description of trajectories of
a control system requires solving a large number of nonlinear
differential equations, which is generally very difficult, if not impossible. 

In past few decades, different approaches have been developed to
study the fundamental properties of small-time locally controllable
systems using their vector fields. The essence of most of these approaches is to provide answers
to two fundamental questions: i) how much pointwise information about the vector
fields of the system is needed to completely characterize small-time local
controllability of the system?, and ii) how is the asymptotic
behaviour of the reachable sets of the system for small times?  It turns out
that these two questions are closely connected and the answers to them would
shed some light on other important questions in mathematical control
theory~\cite{AA:1999, RMB-GS:90,MK:06}. Despite a large body
of literature on this topic, for general control systems, the above questions are still unanswered.

\paragraph*{Literature review} In control literature, various
framework have been proposed for studying nonlinear control
systems~\cite{JPA-ACL:1984, jurdjevicgeometric1996, JCW:1979}. It
turns out that the geometric control theory is one of the suitable
settings for studying controllability of systems. In geometric control
theory, a control system is defined as a \emph{parametrized} family of
vector fields on a \emph{manifold}, where the parameters are the controls and
the manifold is the state space of the
system~\cite{jurdjevicgeometric1996}. For $\nu\in
\mathbb{Z}_{\ge 0}\cup\{\infty,\omega\}$, a $\mathrm{C}^{\nu}$ control-affine
system is defined as a pair $(\mathcal{X},\mathfrak{C})$, where
$\mathcal{X}=\{X_0,X_1,\ldots,X_m\}$ is a family of
$\mathrm{C}^{\nu}$ vector fields on $\R^n$ and $\mathfrak{C}\subseteq
\R^m$ is a control set such that $\vect{0}_m\in \mathfrak{C}$. A trajectory for the control-affine system $(\mathcal{X},\mathfrak{C})$ is
an absolutely continuous curve $x:[0,T]\to \R^n$ such that 
\begin{align*}
\dot{x}(t)= X_0(x(t))+\sum_{i=1}^{m}u_{i}(t)X_i(x(t)),\quad \mbox{
  for almost every } t\in [0,T],
\end{align*}
for some measurable controls $u_1,u_2,\ldots,u_m:[0,T]\to
\mathfrak{C}$. In this paper, we study the system around an equilibrium point $x_0\in
\mathbb{R}^n$, i.e., a point $x_0$ satisfying
$X_0(x_0)=\vect{0}_n$. For a time $t\in \R_{\ge 0}$, the reachable set of $(\mathcal{X},\mathfrak{C})$
form $x_0$ for time less than or equal to $t$, which is denoted by
$\mathrm{R}_{\mathcal{X}}(\le t,x_0)$, is the set of points in state space
$\R^n$ which can be reached by traveling along the trajectories of the
vector fields in $\mathcal{X}$ for positive times less than $t$. More
precisely, 
\begin{align*}
\mathrm{R}_{\mathcal{X}}(\le t,x_0)=\{x(T)\mid x:[0,T]\to \R^n \mbox{
                                      is a trajectory of }
                                      \mathcal{X}, \ x(0)=x_0, \ T\le t\}.
\end{align*}

Among different notions of controllability proposed in the literature,
small-time local controllability is arguably the most fundamental
one. A control system $\mathcal{X}$ is small-time locally controllable
(STLC) from a point $x_0$ if, for every $t>0$, the reachable set
$\mathrm{R}_{\mathcal{X}}(\le t, x_0)$ contains a neighborhood of
$x_0$. Different approaches have been proposed in the literature for characterizing
small-time local controllability using the local
information of the vector fields of the system. The essence of most of
these approaches can be explained using the fundamental result of Nagano~\cite{TN:1966}, which connects
the diffeomorphism invariant properties of a system to the Lie
algebra of its the vector fields (cf.~\cite{YK-TO:16} for an alternative approach to
study small-time local controllability). Using these approaches,
small-time local controllability of systems has been studied in the literature and many
sufficient conditions (cf.~\cite{RH:ADL:2004,AK:1974,sussmanna1978,sussmannlie1983,sus1987})
as well as some necessary conditions (cf.~\cite{MK:1987,MIK:1998,GS:1986,sus1987}) have been
developed. Despite these deep results, in general, the gap between the necessary
and sufficient controllability conditions is large and complete characterization of
small-time local controllability is only possible for some specific classes of
systems (cf.~\cite{COA-ADL:2012,NNP:1976,sussmancontrollability1972}). In what
follows we review some of these ideas and connect them with the
fundamental questions about small-time locally controllable systems.

One of the important notions of controllability which has a close connection with
small-time local controllability is local accessibility. A control system is locally accessible from $x_0$ if
the reachable sets of $\mathcal{X}$ starting from $x_0$ have nonempty
interiors for all positive times. It is clear that if a system is small-time locally controllable
from $x_0$, then it is locally accesible from $x_0$. However, the converse may
not be true~\cite[Example 7.1]{bullogeometric2004}. In 1972, Sussmann
and Jurdjevic characterized the local accessibility of real analytic control systems using the Lie brackets
of their vector fields at the point $x_0$~\cite[Corollary
4.7]{sussmancontrollability1972}. In 1974, Sussmann used an extension of Nagano's
Theorem~\cite{TN:1966} to show that the Lie brackets of the
vector fields of the system also play a crucial role in small-time local
controllability of systems~\cite{sussmann1974extension}. This result
motivated the search for sufficient controllability conditions in terms of Lie brackets of vector fields of the
system~\cite{sussmannlie1983, sus1987}. Later works in this direction
exploit suitable filterations of vector fields to find sharper necessary and sufficient conditions for
small-time local controllability~\cite{HH:91}.

One of the nice features of Sussmann and Jurdjevic's characterization for
local accessibility is that it can be checked using only \emph{finite} number of
differentiations of vector fields of the system at the point $x_0$. This
seemingly trivial observation raises the following important question about the nature
of small-time local controllability: is it possible to characterize small-time local
controllability of a given real analytic system using finite number of
differentiations of its vector fields at the point $x_0$? More precisely, this question can be 
formulated as the following conjecture (see~\cite{AA:1999}).
\begin{conjecture}\label{conj:2}
Given a real analytic control-affine system $\mathcal{X}=\{X_0,X_1,\ldots,X_m\}$ which is small-time locally
controllable from an equilibrium point $x_0$, there exists $N\in\N$ such
that, every real analytic control-affine system
$\mathcal{Y}=\{Y_0,Y_1,\ldots,Y_m\}$ with the property that, for every
$i\in \{0,1,2,\ldots,m\}$, the vector fields $Y_i$ and $X_i$ has the same Taylor polynomial of
order $N$ around $x_0$ is again small-time locally controllable from $x_0$.
\end{conjecture}

Another useful notion for studying small-time local
controllability is the control variation. A control variation can be considered as 
a high-order tangent to the reachable sets of the system which shows the admissible directions in
the reachable sets, i.e., for small times, one will stay inside the
reachable sets by traveling in these directions. By constructing a suitable family of control variations which generates all
the directions in $\R^n$ and using a suitable generalized open mapping theorem,
one can show that a control system is small-time locally
controllable (see, for example~\cite[Theorem 2.1]{HF:1989}). In the control literature,
many different families of control variations have been introduced for
studying small-time local controllability of systems
(cf.~\cite{COA-ADL:2012,bianchinicontrollability1993,
  HF:1987,HF:1989,MK:1990,AJK:1977,sus1987}). The essence of most of
these constructions is to use suitable switchings between vector fields of
the system. Control variations can also be used for studying the rate of growth of the
reachable sets of a system with respect to time. The
order of a control variation reveals how fast one can travel in the
reachable sets in that direction. More specifically, if one can get all
direction in $\R^n$ using families of control variations of order less
than equal to $N$, then there exists a positive constant $C>0$ such that
the closed ball centered at $x_0$ with radius $Ct^N$ (which we 
denote by $\overline{\mathrm{B}}(x_0,Ct^N)$) is contained in the
reachable set  $\mathrm{R}_{\mathcal{X}}(\le t,x_0)$, for small
positive times $t$~\cite[Theorem 2.1]{HF:1989}. This raises the following question: Given a small-time
locally controllable system, does there exist a family of control variations of order $N$ which can be used to prove small-time local controllability
of the system. Motivated by the above question, one can propose the the following conjecture (see~\cite{AA:1999}).
\begin{conjecture}\label{conj:1}
Let $\mathcal{X}$ be a real analytic control-affine system which is small-time
locally controllable from $x_0$. Then there exist $N\in \N$ and
$T,C>0$ such that 
\begin{equation*}
\overline{\mathrm{B}}(x_0,Ct^N)\subseteq \mathrm{R}_{\mathcal{X}}(\le t,
x_0),\qquad\forall t\le T.
\end{equation*}
\end{conjecture}
It turns out that this polynomial growth condition for reachable sets of a
system has a close connection with the regularity of
the time-optimal map of the system~\cite{RMB-GS:90}. One can show that, if the
control system $\mathcal{X}$ is small-time locally controllable, then
the time-optimal map of $\mathcal{X}$ is locally
continuous~\cite{NNP:1977} (cf.~\cite[Theorem 2.2]{RMB-GS:90}, where
this local result has been extended to
a larger domain called escape domain). Similarly, one can show that the polynomial growth
condition for a control system $\mathcal{X}$ is equivalent to local
H\"{o}lder continuity of the time-optimal map of
$\mathcal{X}$~\cite{NNP:1971},~\cite[Theorem 2.5]{RMB-GS:90}.

One can easily check that, if the Conjecture
\ref{conj:1} is true then, for every small-time locally controllable
system, there exists a
family of control variations of order $N$ for the system
which generates all the directions in $\R^n$. As mentioned
in~\cite{AA:1999}, the results in~\cite{NNP:1976} show that both
Conjectures~\ref{conj:2} and~\ref{conj:1} hold on $\R^2$. However, to the best of our knowledge, these two conjectures
are still open for Euclidean spaces $\R^n$ with $n\ge 3$.

One of the challenges for studying Conjectures~\ref{conj:2}
and~\ref{conj:1} stems form the switchings in control variations. Most of the sufficient conditions for
small-time local controllability use control variations
with a finite number of switchings (e.g.~the sufficient
conditions in~\cite{sus1987}). It is well-known that if the family of
the control variations used for proving 
small-time local controllability of the control system have finite
number of switchings, then Conjecture~\ref{conj:2} holds for the
system. In 1988, Kawski found an elegant example of a polynomial control system which is small-time locally
controllable from $x_0$, but it is impossible to check small-time local
controllability using control variations with a finite number of switchings~\cite{MK:88,
  MK:1990}. Kawski used a specific class of control variations with an increasing number of
switching, which he called \emph{fast-switching variations}, to prove
small-time local controllability of this system. This example shows that
more complicated family of variations might be needed to characterize 
small-time local controllability. In~\cite{AAA-RVG:93}, Agrachev
and Gamkrelidze used a detailed analysis of the semigroup of
diffeomorphisms and built a framework for studying small-time local
controllability of control systems using fast-switching variations. Here, we revisit the famous example of
Kawski~\cite{MK:88}, to illustrate the complications that might arise
in studying Conjectures~\ref{conj:2} and~\ref{conj:1} using fast-switching variations.

\begin{expl}\label{ex:1}
Consider the control system $\mathcal{X}$ on $\R^4$, defined by
\begin{eqnarray*}
\dot{x}_1&=&u(t),\\
\dot{x}_2&=&x_1,\\
\dot{x}_3&=&x_1^3,\\
\dot{x}_4&=&x_3^2-x_2^7,
\end{eqnarray*}
where $u:\mathbb{R}\to [-1,1]$ is measurable. We want to study small-time local controllability of $\mathcal{X}$
from $\vect{0}_4\in \R^4$. Using suitable families of control
variations with finite numbers of switchings, one can show that
$\left\{\pm\frac{\partial}{\partial x_1}, \pm\frac{\partial}{\partial
    x_2}, \pm\frac{\partial}{\partial x_3}, \frac{\partial}{\partial
    x_4}\right\}$ are the admissible directions in the reachable set
of the systems $\mathcal{X}$~\cite{MK:88,sus1987}. In order to prove small-time local controllability of $\mathcal{X}$, one needs to find a control variation which generates the direction $-\frac{\partial}{\partial x_4}$.
It can be shown that there is no family of control variations with finite number of
switching which generates the direction $-\frac{\partial}{\partial
  x_4}$~\cite[Claim 2]{MK:88}. However, by using fast-switching variations, one can show that $\mathcal{X}$ is
small-time locally controllable from $\vect{0}_4\in \R^4$~\cite[Claim
1]{MK:88}~(see~\cite{MK:88} and \cite{MK:1990} for the elaborate construction of these control variations). Now consider the
control system $\mathcal{Y}$ on $\R^4$ defined by
\begin{eqnarray*}
\dot{y}_1&=&u(t),\\
\dot{y}_2&=&y_1,\\
\dot{y}_3&=&y_1^3,\\
\dot{y}_4&=&y_3^2-y_2^7+y_1^{58},
\end{eqnarray*}
where $u:\mathbb{R}\to [-1,1]$ is measurable. Note that $\mathcal{X}$ and $\mathcal{Y}$ have the same Taylor polynomial of order $57$ at
$\vect{0}_4\in \R^4$. Using the classical finite switching control
variations, it is easy to show that
$\left\{\pm\frac{\partial}{\partial y_1}, \pm\frac{\partial}{\partial
    y_2}, \pm\frac{\partial}{\partial y_3}, \frac{\partial}{\partial
    y_4}\right\}$ are the admissible directions in the reachable set
$\mathcal{Y}$. However, using the same fast-switching variations as for
the system $\mathcal{X}$, it is very complicated to study whether $-\frac{\partial}{\partial
    y_4}$ is an admissible direction for the reachable sets of the control system
$\mathcal{Y}$ at the point $\vect{0}_4\in \R^4$.
\end{expl}

In general, the families of control variations that are used to prove small-time local
controllability of a system might be even more complicated than the fast-switching control variations in Example~\ref{ex:1}. Therefore, for
small-time locally controllable systems, studying Conjecture
\ref{conj:2} using the form of families of control variations does not
seem to be conclusive.

\paragraph*{Contributions} The contributions of this paper are
manifold. First, we review an operator approach for studying piecewise
constant vector fields and their
flows called chronological calculus. This operator approach, which is
originally introduced in~\cite{agrachev1978exponential}, uses linear
operators on space of smooth functions to estimate the flows of
piecewise constant vector fields (see~\cite[Proposition
  2.1]{agrachev1978exponential} and~\cite[\S 2.4.4]{agrachevcontrol2004}). Using a suitable topology on
the space of real analytic
functions ~\cite{AM:1966}~\cite[\S 1.6 Theorem 27]{PD:2010}, a slightly
different version of this approach has been introduced
in~\cite{SJ:2016} to estimate the flows of real analytic piecewise constant  vector
fields (see~\cite[Theorem 3.8.1]{SJ:2016}). As the first minor contribution of this paper, we use these
frameworks to prove a uniform bound on these estimates of the flows of
piecewise constant vector fields. 

We study the existing literature on small-time local controllability of
systems and review a general class of control
variations defined in~\cite[Definition 2.1]{HF:1989}. As the second minor contribution of this
paper, we use this class of variations to prove the following useful
theorem: for a control system, the cone generated by the control
variations of order $N$ is the space $\R^n$ if and only if the
reachable sets of the control system shrink with polynomial rate of
order $N$ or higher with respect to time (see
  ~\cite[Theorem 1.3]{HF:1987}
  and~\cite[Theorem 2.1]{HF:1989}). 

Next, we introduce a suitable mapping for studying perturbations of reachable sets of real analytic control systems. 
We focus on real analytic control systems whose reachable sets
shrink with polynomial rate of order $N$ or higher. Using the notion
of normal reachability introduced in~\cite{MR0394756} and~\cite[Definition 3.4]{KAG:1992}, we construct a
multi-valued mapping (called perturbation mapping) by composing two
different maps: i) a map from the points in reachable sets of the original control system to the
switching times associated to the control variations, and ii) a
map from the switching times of the control variations to the
reachable sets of the perturbed system. We show that this multi-valued
mapping can capture the effect of perturbations of the vector fields of the system
on its reachable sets. Moreover, we prove the regularity of this perturbation mapping with respect to
time and states of the system. 

Finally, we prove the main result of the paper: for a
real analytic control-affine system whose reachable sets
shrink with polynomial rate of order $N$ or higher with respect to time, small-time local controllability is preserved
under polynomial vector field perturbations of order higher than $N$. The key idea for the proof is to use a
suitable family of control variations for the original system which
generates the space $\mathbb{R}^n$ and employ the real analytic version of
chronological calculus to show that the perturbed family of control variations
also generates the space $\mathbb{R}^n$. In particular, we show that our main
result in this paper implies that if Conjecture~\ref{conj:1} holds
then Conjecture~\ref{conj:2} also holds.

\paragraph*{Paper Organization}
In Sections~\ref{sec:notation} and~\ref{sec:functions_vector_fields}, we introduce the essential notation
for stating the main results of the paper.  In
Section~\ref{sec:chronological}, we introduce an operator approach for studying
piecewise constant vector fields and their flows. In
Section~\ref{sec:control_systems}, we introduce and study different
notions of controllability for $C^{\nu}$ control-affine systems. In Section~\ref{sec:variations} the notion of 
control variations is defined and a characterization of the growth rate condition is presented based on
the control variations of the system. In
Section~\ref{sec:perturbations}, for every two real analytic control systems
$\mathcal{X}$ and $\mathcal{Y}$ and every time $t$, we construct a
multi-valued mapping between the reachable sets of
the control system $\mathcal{X}$ and $\mathcal{Y}$. Finally, in Section~\ref{sec:main_result}, the main result of this paper is stated and proved.

\subsection{Notations and conventions}\label{sec:notation}
In this paper, the set of integers, non-negative integers, and natural
numbers are denoted by $\Z$, $\Z_{\ge 0}$, and $\N$, respectively. We denote the $n$-dimensional Euclidean space by
$\R^{n}$ and the zero vector in $\R^n$ by $\vect{0}_n$. The Euclidean norm of a vector $\mathbf{v}$ in $\R^n$ is
denoted by $\|\mathbf{v}\|$. The $n$-sphere is denoted by
$\mathbb{S}^n$ and the non-negative orthant in $\mathbb{R}^n$ is
denoted by $\R_{\ge 0}^{n}$. We denote
\begin{align*}
[-1,1]^n=\underbrace{[-1,1]\times [-1,1]\times \ldots\times [-1,1]}_{n
  \small\mbox{ times }}.
\end{align*}
For a nonempty subset $S\subseteq \R^n$, the interior of $S$ in $\R^n$ is
denoted by $\mathrm{int}(S)$ and the closure of $S$ in $\R^n$ is denoted by
$\overline{S}$. A multi-index of order $m$ is an element
$\mathbf{r}=(r_1,r_2,\ldots,r_m)\in \Z_{\ge 0}^m$. For all multi-indices $\mathbf{r}$ and $\mathbf{s}$ of order $m$, every
$\mathbf{x}=(x_1,x_2,\ldots,x_m)\in \R^m$, and $f:\R^m\to\R^n$, we define
\begin{eqnarray*}
|\mathbf{r}|&=&r_1+r_2+\ldots+r_m,\qquad\quad \mathbf{r}!=(r_1!)(r_2!)\ldots(r_m!),\\
\mathbf{x}^{\mathbf{r}}&=&x_1^{r_1}x_2^{r_2}\ldots x_m^{r_m},\qquad\quad\qquad
D^{\mathbf{r}}f(x)=\frac{\partial^{|\mathbf{r}|} f}{\partial
               x_1^{r_1}\partial x_2^{r_2}\ldots \partial x_m^{r_m}}.
\end{eqnarray*}
The space of all decreasing sequences $\{a_i\}_{i\in\N}$ such that
$a_i\in \R_{>0}$ and $\lim_{n\to\infty} a_n=0$ is denoted by
$c^{\downarrow}_{0}$. Let $x\in \R^n$ and $r\in
\R_{>0}$. Then the Eucleadian open ball centered at $x$ with radius $r$ is
denoted by $\mathrm{B}(x,r)$ and its closure is denoted by
$\overline{\mathrm{B}}(x,r)$. In this paper, whenever we use the letter $\nu$, we mean that $\nu\in
\N\cup\{\infty,\omega\}$. Let $U\subseteq\R^m$ be an open set and
$f:U\to \R^{n}$. For
$\nu\in N\cup\{\infty\}$, the
mapping $f$ is a $\mathrm{C}^{\nu}$-mapping if, for every multi-index
$\mathbf{r}\in \Z_{\ge 0}^m$ with property that $|\mathbf{r}|\le \nu$,
the mapping $D^{\mathbf{r}}f$ is continuous. The
mapping $f$ is a $\mathrm{C}^{\omega}$-mapping if it is a
$\mathrm{C}^{\infty}$-mapping and, for every $x_0\in U$, the
Taylor series of $f$ around $x_0$ converges locally. Let $k\in \N$, $(V,\|.\|_{V})$ be a normed vector space, and $f:\R\to V$ and
$g:\R\to V$ be two curves on $V$. Then we write
\begin{align*}
f(x)=g(x)+\mathcal{O}(x^k)
\end{align*}
if there exists $\alpha\in \R$ such that we have $\lim_{x\to 0}\frac{\|f(x)-g(x)\|_{V}}{|x|^{k}}=\alpha$. Let $U$ and $V$ be
two sets and $F:U \rightrightarrows V$ be a multi-valued map. Then
a selection of $F$ is a single-valued mapping $f:U \to
V$ with the property that $f(x)\in F(x)$, for every $x\in U$.

\section{Functions and vector fields}\label{sec:functions_vector_fields}

In this section, we study functions and vector fields on the
Euclidean space $\R^n$. The space of all $\mathrm{C}^{\nu}$-functions on $\R^n$ is denoted by
$\mathrm{C}^{\nu}(\R^n)$ and the space of all $\mathrm{C}^{\nu}$-vector fields on
$\R^n$ is denoted by $\Gamma^{\nu}(\R^n)$. It is easy to see that
both $\mathrm{C}^{\nu}(\R^n)$ and $\Gamma^{\nu}(\R^n)$ are vector
spaces over $\R$. Given $x_0\in \R^n$, we define the functional
$\mathrm{ev}_{x_0}:\mathrm{C}^{\nu}(\R^n)\to \R$ by $\mathrm{ev}_{x_0}(f)=f(x_0)$,
for every $f\in \mathrm{C}^{\nu}(\R^n)$. In control theory, it is common to work with
time-varying vector fields. In this paper, without loss
  of generality, we restrict our attention to piecewise constant
  vector fields.

\begin{definition}[\textbf{Piecewise constant vector fields}]
Let $\mathbb{T}\subseteq \R$ be an interval. The map $X:\mathbb{T}\times \R^n\to \R^n$ is a
piecewise constant vector field of class $\mathrm{C}^{\nu}$ if the
following hold:
\begin{itemize}
\item[(i)] For every $t\in \mathbb{T}$, the map $X_t:\R^n\to \R^n$ defined by
\begin{equation*}
X_t(x)=X(t,x),\qquad\forall x\in \R^n,
\end{equation*}
is a vector field of class $\mathrm{C}^{\nu}$.
\item[(ii)] For every $x\in \R^n$, the map $X^x:\mathbb{T}\to \R^n$ defined
  by
\begin{equation*}
X^x(t)=X(t,x),\qquad\forall t\in \mathbb{T},
\end{equation*}
is piecewise constant. 
\end{itemize}
\end{definition}
Let $X:\R\times\R^n\to\R^n$ be a piecewise constant vector
field. Then, by the fundamental theorem of differential
equations~\cite[Theorem 2.3]{coddington1955theory}, for every $x_0\in
\R^n$, there exist a maximal interval $\mathbb{T}_{x_0}$ and an absolutely continuous curve $t\mapsto
\exp(tX)(x_0)$ which satisfies the following initial value
problem:
\begin{align*}
\frac{d}{dt}(\exp(tX)(x_0))&=X(t,\exp(tX)(x_0)),\qquad\mbox{ for almost
                              every }
                              t\in \mathbb{T}_{x_0}\\
\exp(0 X)(x_0)&=x_0.
\end{align*}
The map $t\mapsto \exp(tX)(x_0)$ is called the integral curve of the
piecewise constant vector field
$X$ passing through $x_0$.

\begin{definition}[\textbf{Complete vector fields}]
A piecewise constant vector field $X:\R\times \R^n\to \R^n$ is complete if, for every $x_0\in \R^n$, the integral
curve of $X$ passing through $x_0$ exists for all $t\in\R$.
\end{definition}


\begin{definition}[\textbf{Flows of piecewise constant vector fields}]
Let $X:\R\times\R^n\to\R^n$ be a complete piecewise constant $C^{\nu}$-vector
field. Then the flow of $X$ is the map $\exp(X):\R\times \R^n\to \R^n$
defined by
\begin{align*}
\exp(X)(t,x)=\exp(tX)(x),\qquad\forall (t,x)\in \R\times \R^n.
\end{align*}
\end{definition}

\section{Operator approach for time-varying vector fields}\label{sec:chronological}

In this section, we review a well-known operator approach
which allows us to translate the nonlinear finite-dimensional systems
into linear infinite-dimensional systems. This
operator approach, which is known as \emph{chronological calculus},
was first proposed by Agrachev and Gamkrelidze
in~\cite{agrachev1978exponential}. While the chronological calculus 
is originally developed to study time-varying vector fields, in this paper we
focus on a simpler version of it which studies piecewise constant vector fields. In this framework,
$\mathrm{C}^{\infty}$-vector fields and
$\mathrm{C}^{\infty}$-diffeomorphisms are identified with derivations
and unital algebra isomorphisms on $\mathrm{C}^{\infty}(\R^n)$,
respectively. These identifications are then used to recast the nonlinear dynamical system governing
the flow of a piecewise constant vector field to a linear differential equation an infinite-dimensional space. Using this recasting and the Whitney
compact-open topology on the space of $\mathrm{C}^{\infty}(\R^n)$, an
asymptotic expansion for the flow of a piecewise constant real
analytic vector field is developed and its convergence has been
studied in~\cite{agrachev1978exponential}. In~\cite{SJ-ADL:2014}
and~\cite{SJ:2016}, this framework has been extended in two directions. First, the real analytic vector fields
are considered as derivations on $\mathrm{C}^{\omega}(\R^n)$
and real analytic diffeomorphisms are considered as unital algebra homomorphism on
$\mathrm{C}^{\omega}(\R^n)$. Moreover, the space
$\mathrm{C}^{\omega}(\R^n)$ is endowed with the $\mathrm{C}^{\omega}$-topology
and the convergence of the asymptotic expansion for the flow of a
piecewise constant real analytic vector field is studied
in this new topology~\cite{SJ:2016}. In the sequel, we adopt the
approach of~\cite{SJ-ADL:2014} for studying piecewise constant real analytic vector fields and their flows.

\begin{definition}[\textbf{Vector fields and diffeomorphisms}]\label{def:vector_field_derivation}
Let $V:\R^n\to \R^n$ be a real analytic vector field and $\phi:\R^n\to
\R^n$ be a real analytic diffeomorphism. Then 
\begin{enumerate}
\item[(i)] we define the derivation $\widehat{V}:\mathrm{C}^{\omega}(\R^n)\to \mathrm{C}^{\omega}(\R^n)$ by
\begin{equation*}
\widehat{V}(f)=\mathcal{L}_{V}f,
\end{equation*}
where $\mathcal{L}_{V}f$ is the Lie derivative of $f$ in the
direction of the vector field $V$.
\item[(ii)] we define the unital algebra homomorphism $\widehat{\phi}:\mathrm{C}^{\omega}(\R^n)\to \mathrm{C}^{\omega}(\R^n)$ by
\begin{equation*}
\widehat{\phi}(f)=f\scirc\phi.
\end{equation*}
\end{enumerate}
\end{definition}
The space of linear mappings from $\mathrm{C}^{\omega}(\R^n)$ to $\mathrm{C}^{\omega}(\R^n)$ is denoted by
$\mathrm{LC}^{\omega}(\R^n)$ and it is clear that we
have $\Gamma^{\omega}(\R^n)\subset \mathrm{LC}^{\omega}(\R^n)$. Thus, for every real
analytic vector field $X:\R^n\to \R^n$ and
every real analytic mapping $\phi:\R^n\to \R^n$, we have $\widehat{\phi},\widehat{V}\in \mathrm{LC}^{\omega}(\R^n)$.
Using the operator characterization of vector fields, a piecewise constant $\mathrm{C}^{\nu}$-vector field $X:\mathbb{T}\times
\R^n\to \R^n$ can be considered as a piecewise constant
curve $t\mapsto \widehat{X}_t$ on the space
$\mathrm{LC}^{\omega}(\R^n)$. Therefore, for studying properties of
piecewise constant vector fields in this framework, we need to define a suitable topology on the vector space $\mathrm{LC}^{\omega}(\R^n)$.
\begin{definition}[\textbf{Topological vector space}]
Let $E$ be a vector space over $\R$ and $\tau$ be a topology on
$E$. Then 
\begin{enumerate}
\item[(i)] the pair $(E,\tau)$ is a topological vector space if both addition and scalar
multiplication in $E$ are continuous with respect to $\tau$; 
\item[(ii)] the topological vector space $(E,\tau)$ is a locally convex space
  if the topology $\tau$ is generated by a family of seminorms $\{p_{i}\}_{i\in\Lambda}$ on $E$;
\item[(iii)] a subset $B\subseteq E$ is bounded with respect to $\tau$ if, for every neighborhood $U$of $0$ in $E$,
there exists $\alpha\in \R$ such that $B\subset \alpha U$.
\end{enumerate} 
\end{definition}
For locally convex spaces, bounded sets can be equivalently
characterized using the seminorms~\cite[Theorem 1.37]{rudinfunctional}. 
\begin{theorem}[\textbf{Seminorm characterization of bounded sets}]\label{thm:boundedset_locallyconvex}
Let $E$ be a locally convex space which is generated by the family of
seminorms $\{p_{i}\}_{i\in\Lambda}$. A set $B\subseteq E$ is bounded
if and only if, for every $i\in\Lambda$, there exists $N_i\in \R_{>0}$
such that
\begin{align*}
p_i(v)\le N_i,\qquad\forall v\in B.
\end{align*}
\end{theorem}

We are now ready to define a locally convex topology on the vector
spaces $\mathrm{C}^{\omega}(\R^n)$ and $\mathrm{LC}^{\omega}(\R^n)$ using families of seminorms. 

\begin{definition}[\textbf{Real analytic seminorms}]
Let $K\subset \R^n$ be a compact set and $\mathbf{a}\in
c^{\downarrow}_0$. 
\begin{enumerate}
\item[(i)] We define the seminorm $\rho^{\omega}_{K,\mathbf{a}}:\mathrm{C}^{\omega}(\R^n)\to \R_{\ge 0}$ as
\begin{equation*}
\rho^{\omega}_{K,\mathbf{a}}(f)=\sup\left\{\frac{a_0a_1\ldots a_{|r|}}{|r|!}\left\|
    D^{(r)}f(x)\right\|\suchthat x\in K, |r|\in \Z_{\ge 0}\right\}.
\end{equation*}
The topology on $\mathrm{C}^{\omega}(\R^n)$ generated by the family of seminorms
$\rho^{\omega}_{K,\mathbf{a}}$ is called the $\mathrm{C}^{\omega}$-topology.
\item[(ii)] Let $f\in \mathrm{C}^{\omega}(\R^n)$. We define the seminorm
$\rho^{\omega}_{K,\mathbf{a},f}:\mathrm{LC}^{\omega}(\R^n)\to \R_{\ge 0}$ as
\begin{equation*}
\rho^{\omega}_{K,\mathbf{a},f}(X)=\sup\left\{\frac{a_0a_1\ldots a_{|r|}}{|r|!}\left\|
    D^{(r)}\left(Xf\right)(x)\right\|\suchthat x\in K, |r|\in \Z_{\ge 0}\right\}
\end{equation*}
The topology on $\mathrm{LC}^{\omega}(\R^n)$ generated by the family of seminorms
$\left\{\rho^{\omega}_{K,\mathbf{a},f}\right\}$ is
called the $\mathrm{C}^{\omega}$-topology.
\end{enumerate}
\end{definition}

Note that one can define another locally convex topology on
$\mathrm{C}^{\omega}(\R^n)$ by inducing the Whitney compact-open
topology and using the subspace relation~\cite[\S 2.2]{agrachevcontrol2004},~\cite[\S
6]{kriegl1997convenient}. It turns out that the $\mathrm{C}^{\omega}$-topology on the space
$\mathrm{C}^{\omega}(\R^n)$ is finer than the subspace topology
induced from the Whitney topology on
$\mathrm{C}^{\infty}(\R^n)$~\cite[Chapter 5]{SJ-ADL:2014}. The $\mathrm{C}^{\omega}$-topology and its properties has been studied throughly
in~\cite{AM:1966},~\cite{PD:2010}, and~\cite{kriegl1997convenient}. The
$\mathrm{C}^{\omega}$-topology on the space of real analytic functions
has been first defined and studied using advanced tools in analysis in~\cite{AM:1966}. The above seminorm characterization
of the $\mathrm{C}^{\omega}$-topology has been introduced and proved in~\cite{DV:2013} (see~\cite{PD:2010}
for a detailed study of the $\mathrm{C}^{\omega}$-topology on the space $\mathrm{C}^{\omega}(\R^n)$ ). Using the $\mathrm{C}^{\omega}$-topology on the vector space
$\mathrm{LC}^{\omega}(\R^n)$, we can study properties of
piecewise constant vector fields. The set of
piecewise constant curves with domain $\mathbb{T}$ on $\mathrm{LC}^{\omega}(\R^n)$
is denoted by $\mathrm{PC}(\mathbb{T};\mathrm{LC}^{\omega}(\R^n))$. Let $S\subseteq \mathrm{LC}^{\omega}(\R^n)$. We define the subset
$\mathrm{PC}(\mathbb{T};S)\subseteq \mathrm{PC}(\mathbb{T};\mathrm{LC}^{\omega}(\R^n))$ as
\begin{equation*}
\mathrm{PC}(\mathbb{T};S)=\left\{\lambda\in
  \mathrm{PC}(\mathbb{T};\mathrm{LC}^{\omega}(\R^n))\suchthat \lambda(t)\in S, \
  \mbox{ for almost every } t\in \mathbb{T} \right\}.
\end{equation*} 

Let $X:\mathbb{T}\times \R^n\to \R^n$ be a piecewise constant real analytic
vector field. Then it is easy to see that $X$ is
piecewise constant if and only if $t\mapsto
\widehat{X}_t$ is a piecewise constant curve on
$\mathrm{LC}^{\omega}(\R^n)$. By considering $X$ as a curve $t\mapsto
\widehat{X}_t$ on the space $\mathrm{LC}^{\omega}(\R^n)$, one can also translate the nonlinear differential equations governing the flow of $X$:
\begin{align*}
\frac{d}{dt}\exp(tX)(x_0)&=X(t,\exp(tX)(x_0)),\quad \mbox{ for almost
                           every }t\in \R\\
\exp(0X)(x_0)&=x_0,
\end{align*}
into the following linear differential equations:
\begin{eqnarray}\label{eq:1}
\begin{split}
\frac{d}{dt}\widehat{\exp}(tX)&=\widehat{\exp}(tX)\scirc
\widehat{X}_t,\quad \mbox{ for almost every }t\in\R\\
\widehat{\exp}(0X)&=\mathrm{id},
\end{split}
\end{eqnarray}
where $\widehat{\exp}(tX)$ is the unital algebra homomorphism
associated to $\exp(tX)$ (see
Definition~\ref{def:vector_field_derivation}). Note that
equation~\eqref{eq:1} is a family of linear differential equations on
the infinite dimensional locally convex space $\mathrm{LC}^{\omega}(\R^n)$. 
One can study the sequence of Picard iterations for this infinite dimensional linear
differential equations~\eqref{eq:1}~\cite[Chapter 1, \S
3]{coddington1955theory},~\cite{SGL-GOS:1994}. 
\begin{definition}[\textbf{Sequence of flow iterations}]\label{eq:sequence_iterations}
Let $X$ be a piecewise constant vector field of class
$\mathrm{C}^{\omega}$. We define the curve 
$t\mapsto \widehat{\exp}_0(tX)$ on $\mathrm{LC}^{\omega}(\R^n)$ as:
\begin{equation*}
\widehat{\exp}_0(tX)=\mathrm{id},\qquad\forall t\in [0,T'].
\end{equation*}
Then, for every $k\in \N$, we define the curve $t\mapsto
\widehat{\exp}_k (tX)$ on $\mathrm{LC}^{\omega}(\R^n)$ inductively as
\begin{equation}\label{eq:Picards_iterations}
\widehat{\exp}_k(tX)=\mathrm{id}+\int_{0}^{t}\widehat{\exp}_{k-1}(\tau
X)\scirc\widehat{X}(\tau)d\tau,\qquad\forall t\in [0,T'].
\end{equation}
\end{definition}
For linear differential equations on infinite dimensional locally convex
spaces, there does not exist a general result for convergence of the sequence of
Picard iterations~\cite{SGL-GOS:1994}. However, for the differential
equations~\eqref{eq:1}, one can prove the following estimates for the seminorms of the sequence of 
iterations of the flows in
Definition~\ref{eq:sequence_iterations}~\cite[Theorem
3.8.1]{SJ:2016}. The following theorem can be considered as an extension of the
estimates in~\cite[\S 2.4.4]{agrachevcontrol2004}.

\begin{theorem}[\textbf{Estimates for flow iterations}]\label{th:3}
Let $B$ be a bounded set in $\Gamma^{\omega}(\R^n)$. Then the
following statements hold:
\begin{enumerate}
\item[(i)] there exists $T_{B}$ such that, for every $X\in
  \mathrm{PC}([0,T_B];B)$ and every $k\in \N$, the map $t\mapsto
  \widehat{\exp}_k(tX)$ is defined on $[0,T_B]$,
\item[(ii)] for every compact set $K\subseteq \R^n$, every $f\in \mathrm{C}^{\omega}(\R^n)$, and every
$\mathbf{a}\in c^{\downarrow}_0$, there exist positive constants $M, M_f>0$ such that, for every $X\in \mathrm{PC}([0,T_B];
B)$, we have
\begin{equation*}
\rho^{\omega}_{K,\mathbf{a},f}(\widehat{\exp}_k(t X)-\widehat{\exp}_{k-1}(t X))\le
\left(Mt\right)^{k+1} M_f, \quad\forall t\in [0,T_B], \ \forall k\in
\N.
\end{equation*}
\end{enumerate}
\end{theorem}

Using the estimate in Theorem~\ref{th:3}, one can get an estimate for the flow of a vector
field $X$ using the sequence of iterations in Definition~\ref{eq:sequence_iterations}.

\begin{theorem}\label{cor:estimate_picard_iterations}
Let $B$ be a bounded set in $\mathrm{LC}^{\omega}(\R^n)$. Then there
exist $M,L>0$ and $\overline{T}\le T_B$ such that, for every $X\in \mathrm{PC}([0,\overline{T}];
B)$ and every $i\in \{1,2,\ldots,n\}$, we have
\begin{equation*}
\|\mathrm{ev}_{x_0}\scirc\widehat{\exp}(t
X)(x^i)-\mathrm{ev}_{x_0}\scirc\widehat{\exp}_{k}(t X)(x^i)\|\le
\frac{Mt^{k+1}}{1-Mt}L, \quad\forall t\in [0,\overline{T}],
\end{equation*}
where $x^i$ is the $i$th coordinate function on $\R^n$.
\end{theorem}

It is worth mentioning that, Theorem~\ref{th:3}(ii) and Theorem~\ref{cor:estimate_picard_iterations} can alternatively be proved
using the estimates in ~\cite[\S 2.4.4]{agrachevcontrol2004}. However, the real analyticity
  of the time varying vector fields is essential for these results to
  hold. Since Theorem~\ref{cor:estimate_picard_iterations} is crucial
for the proof of the main result of this paper, we provide a proof for it using Theorem~\ref{th:3} in
Appendix~\ref{app:1}. 




\section{Control-affine systems}\label{sec:control_systems}

In this section, we introduce several controllability notions
associated with a $C^{\nu}$ control-affine system
$(\mathcal{X},\mathfrak{C})$. For the rigorous definition of the
$C^{\nu}$ control-affine systems, their trajectories, and their
equilibrium points, we refer the readers to the introduction of the
paper.

\begin{definition}[\textbf{Controllability of $\mathrm{C}^{\nu}$ control-affine
system}]\label{def:reachable}
Suppose that $(\mathcal{X},\mathfrak{C})$ is a $\mathrm{C}^{\nu}$ control-affine
system on $\R^n$, with $\mathcal{X}=\{X_0,X_1,\ldots,X_m\}$, $t\in
\R_{>0}$, and $x_0\in \R^n$ is an equilibrium point of $(\mathcal{X},\mathfrak{C})$. Then
\begin{enumerate}
\item[(i)]
The $\mathrm{C}^{\nu}$ control-affine system $\mathcal{X}$ is small-time locally controllable from
$x_0$ if, for every $t\in \R_{>0}$, we have
\begin{equation*}
x_0\in \mathrm{int}\left(\mathrm{R}_{\mathcal{X}}(\le t, x_0)\right);
\end{equation*} 
\item[(ii)] Let $N\in \Z_{>0}$ be a positive integer. Then the
  $\mathrm{C}^{\nu}$ control-affine system $\mathcal{X}$ satisfies growth
  rate condition of order $N$ at the point $x_0$ if there exist $C, T>0$ such that, for every $t\in (0,T]$, we have
\begin{equation*}
\overline{\mathrm{B}}(x_0, Ct^N)\subset \mathrm{R}_{\mathcal{X}}(\le t, x_0).
\end{equation*} 
\end{enumerate}
\end{definition}
\begin{rem}
The following remarks are in order. 
\begin{enumerate}[label=(\roman*)]
\item Note that, in the definition of the above controllability notions, we do not impose any restriction on the
  structure of the control set $\mathfrak{C}$. For control-affine systems, this
  choice of control set seems natural. However, more general structures
  for the control set, such as separable
  topological space~\cite{sontag1998mathematical} and Frechet spaces~\cite{sussmann1998introduction},
  have been proposed in the literature. In the bundle view of the control
  systems, the control set is usually banished and the system is
  considered as a family of vector fields~\cite{JCW:1979}. In the
  chronological calculus, a control system is usually defined as a
  parametrized family of vector fields on
  $\R^n$~\cite{agrachevcontrol2004}. We refer the interested readers
  to~\cite{SJ-ADL:2014-2} for a review of the literature on various
  definitions of control systems and, in particular,  the properties of
  the real analytic control-affine systems.
  
\item if $x_0$ is an equilibrium point for
  $(\mathcal{X},\mathfrak{C})$, then in the absence of 
  control inputs, i.e., $\mathbf{u}=\vect{0}_m$, the states of the
  system stay at the point $x_0$; 

\item while the definition of reachable sets with measurable controls is
  widely used and studied in the literature (cf.~\cite{MK:1990,sussmannlie1983,sus1987}), some authors consider other classes of
  controls such as piecewise constant controls~\cite{KAG:1992} and bang-bang
  controls~\cite{AK:1974} for
  studying control systems and their reachable sets. In general,
  the reachable sets defined using these different classes of controls
  are not the same~\cite{sus1987}. We refer the interested reader
  to~\cite{KAG:1992,AK:1974,sus1987} for a through study of the 
  connections between these reachable sets; 
  
\item it is clear form Definition~\eqref{def:reachable} that, if a control-affine system $(\mathcal{X},\mathfrak{C})$
  satisfies the growth rate condition of order $N$ at the point $x_0$, then it is small-time
  locally controllable from $x_0$; 
  
\item for a general $C^{\nu}$ control-affine system
  $(\mathcal{X},\mathfrak{C})$, small-time local
  controllability at the point $x_0$ does not necessarily imply that $x_0$ is
  an equilibrium point of the system. However, for a real analytic control-affine
  system $(\mathcal{X},\mathfrak{C})$ with compact control set
  $\mathfrak{C}$, if the system is small-time locally
  controllable from $x_0$, then $x_0$ is in the convex hull of the
  set $\{X_0(x_0)+\sum_{i=1}^{m}u_iX_i(x_0)\mid
  (u_1,\ldots,u_m)\in\mathfrak{C}\}$~\cite[Proposition
  6.1]{sussmannlie1983}. It is worth mentioning that while a large body of the
  research on controllability is devoted to studying small-time local
  controllability of control systems at equilibrium
  points~\cite{MK:1990,sussmanna1978,sussmannlie1983,sus1987}, there
  are some deep and interesting results for small-time local controllability
  along non-stationary reference trajectories~\cite{bianchinicontrollability1993,HH:76}. 
\end{enumerate}
\end{rem}

While the trajectories of a control-affine
systems are constructed using measurable
controls $u:[0,t]\to \mathfrak{C}$, it is sometimes useful to work with
piecewise constant controls and their associated vector fields. 
\begin{definition}[\textbf{Piecewise constant control vector fields}]
Let $(\mathcal{X},\mathfrak{C})$ be a
$\mathrm{C}^{\nu}$ control-affine system with
$\mathcal{X}=\{X_0,X_1,\ldots,X_m\}$ and let 
$\mathbf{u}=(u_1,u_2,\ldots,u_m)\in \mathfrak{C}$. Then the vector
field of the control-affine system $\mathcal{X}$ associated to
$\mathbf{u}$ is $X_{\mathbf{u}}\in \Gamma^{\nu}(\R^n)$ defined by
\begin{align*}
X_{\mathbf{u}}=X_0+\sum_{i=1}^{m} u_iX_i.
\end{align*}
Let $p\in N$, $\mathbf{I}=(\mathbf{u}^1,\mathbf{u}^2,\ldots,\mathbf{u}^p)\in
\mathfrak{C}^p$ be a $p$-tuple of constant controls,
$\mathbf{t}=(t_1,\ldots,t_p)\in \R_{>0}^p$ be a $p$-tuple of
switching times. We define the piecewise constant
control vector field $X^{\mathbf{I},\mathbf{t}}$ by
\begin{equation*}
X^{\mathbf{I},\mathbf{t}}=
\begin{cases}
X_{\mathbf{u}^p}, & t\in [0,t_p],\\
X_{\mathbf{u}^{p-1}}, & t\in (t_{p},t_{p-1}+t_p],\\
\vdots & \vdots\\
X_{\mathbf{u}^1}, & t\in (t_2+\ldots,t_{p}, t_1+\ldots+t_p].
\end{cases}
\end{equation*}
\end{definition}
It clear that, for every $p$-tuple
$\mathbf{I}\in \mathfrak{C}^p$ and every $p$-tuple 
$\mathbf{t}\in\R^p_{>0}$ the vector field $X^{\mathbf{I},\mathbf{t}}$
is piecewise constant and the following property holds:
\begin{equation*}
\exp(|\mathbf{t}|X^{\mathbf{I},\mathbf{t}})(x)=\exp(t_1X_{\mathbf{u}^1})\scirc
\exp(t_2X_{\mathbf{u}^2})\scirc
\ldots\scirc\exp(t_pX_{\mathbf{u}^p})(x),\qquad\forall x\in \R^n.
\end{equation*}

Another notion relevant to small-time local controllability
of systems is normal reachability \cite{KAG:1992}. Normal reachability
has been first introduced and studied by Sussmann in \cite{HJS:1976}.
\begin{definition}[\textbf{Normal reachability}]
Let $(\mathcal{X},\mathfrak{C})$ be a $\mathrm{C}^{\nu}$ control-affine
system on $\R^n$ with $\mathcal{X}=\{X_0,\ldots,X_m\}$ and let
$x_1,x_0\in \R^n$. Then the point $x_1$ is normally reachable in time
less than $t$ from $x_0$, if the following conditions hold:
\begin{enumerate}
\item there exist $p\in \N$,
  $\mathbf{u}^1,\mathbf{u}^2,\ldots,\mathbf{u}^p\in \mathfrak{C}$, and $(s_1,s_2,\ldots,s_p)\in \R^p_{>0}$ such
that $s_1+s_2+\ldots+s_p<t$ and
\begin{equation*}
\exp(s_1 X_{\mathbf{u}^1})\scirc \exp(s_2 X_{\mathbf{u}^2})\scirc\ldots\scirc \exp(s_p X_{\mathbf{u}^p})(x_0)=x_1,
\end{equation*}
and 
\item there exists an open neighborhood of $(s_1,s_2,\ldots,s_p)$ in
$\R^p_{>0}$ such that the map
\begin{equation*}
(t_1,t_2,\ldots,t_p)\mapsto \exp(t_1 X_{\mathbf{u}^1})\scirc \exp(t_2 X_{\mathbf{u}^2})\scirc\ldots\scirc \exp(t_p X_{\mathbf{u}^p})(x_0)
\end{equation*}
is $\mathrm{C}^1$ and of rank $n$ on this neighborhood.
\end{enumerate}
\end{definition}
It is clear that, if the point $x_0$ is normally reachable from itself, the system is small-time locally
controllable from $x_0$. However, the converse is not true for
general control-affine systems~\cite[Example 3.9]{KAG:1992}. For real analytic systems, the connection between small-time local controllability
and normal reachability has been studied in \cite{KAG:1992}. In fact,
in \cite{KAG:1992}, it has been shown that for real analytic control
systems, small-time local controllability from $x_0$ implies that, for every time $t$, every point in the interior of the reachable set from
$x_0$ in times less than $t$ is normally reachable from $x_0$ \cite[Theorem 5.5 and Corollary 4.15]{KAG:1992}.
\begin{theorem}\label{th:1}
Let $(\mathcal{X},\mathfrak{C})$ be a  real analytic control-affine
system on $\R^n$ with $\mathcal{X}=\{X_0,X_1,\ldots,X_m\}$. If
$\mathcal{X}$ is small-time locally controllable from $x_0$ then, for every $t>0$, every point in the set
$\mathrm{int}\left(\mathrm{R}_{\mathcal{X}}(\le t,x_0)\right)$ is normally
  reachable in time less than $t$ form $x_0$.
\end{theorem}

Finally, we introduce the following assumption on the class of
$C^{\nu}$ control-affine systems. This assumption consists of two
parts: i) an assumption on the vector fields of the system, and ii) an
assumption on the control set of the system.

\begin{assumption}\label{asump}
We assume that the $C^{\nu}$ control-affine system
$(\mathcal{X},\mathfrak{C})$ satisfies the following condition: 
\begin{enumerate}[label=(\roman*)]
\item\label{p1:complete} for every $p\in \mathbb{Z}_{>0}$
    and every $p$-tuples $\mathbf{I}=(\mathbf{u}^1,\ldots,\mathbf{u}^p)\in\mathfrak{C}^p$ and
  $\mathbf{t}=(t^1,\ldots,t^p)\in \mathbb{R}^p_{>0}$, the piecewise constant
  control vector field $X^{\mathbf{I},\mathbf{t}}$ is complete;
\item\label{p2:unit} the control set $\mathfrak{C}\subseteq \R^m$ is the compact
  convex set $[-1,1]^m$.
\end{enumerate}
\end{assumption}
\begin{rem}In this remark we elaborate on each part of Assumption~\ref{asump}.
\begin{enumerate}[label=(\roman*)]
\item By considering convex and compact control sets, the Assumption~\ref{asump}\ref{p2:unit} is
  restrictive for studying small-time local controllability. It turns
  out that small-time local
  controllability of systems strongly depends on the structure of the control set
  (see~\cite{SJ-ADL:2014-2} for a detailed study of the role of
  control sets in fundamental properties of the systems);
\item If we assume that the control set $\mathfrak{C}$ is compact,
  then the completeness of vector fields in 
  Assumption~\ref{asump}\ref{p1:complete} is not
  restrictive for studying small-time local controllability. This is
  because of the fact that, for studying small-time local controllability of the system, one can
  always focus on a small neighborhood of $x_0$ and small
  times $t$, where the flows of all the piecewise constant vector fields of
  the system exist due to the compactness of $\mathfrak{C}$~\cite[Chapter 2, Theorem 1.1]{coddington1955theory}.
\end{enumerate}
\end{rem}


\section{Control variations}\label{sec:variations}

The notion of control variation is one of the fundamental tools in studying
reachable sets of control systems. Roughly speaking control variations
can be considered as the directions constructed using the trajectories
of the system, along which one can steer the
control system. By constructing appropriate control variations and
using a suitable open mapping theorem, one can show that a control system
is small-time locally controllable~\cite[Theorem 2.1]{HF:1989}. The first use of the notion of control variations for approximating reachable sets of control systems
can be traced back to the original work of Pontryagin and his
coworkers for studying the boundary of reachable sets~\cite{LSP:1964}. Since then, many different and technical notions of variations with
various properties have been proposed in the control
literature~\cite{bianchinicontrollability1993,
  HF:1989,MK:1990,AJK:1977}. In this section we study a
  general class of control variations introduced in \cite{HF:1987,HF:1989}. 
\begin{definition}[\textbf{Control variations}]\label{def:control_variations}
Let $k\in \N$ and $(\mathcal{X},\mathfrak{C})$ be a $\mathrm{C}^{\nu}$ control-affine system on $\R^n$. Then a vector $\mathbf{v}\in \R^n$ is called a variation
of $k$th order for the control-affine system $\mathcal{X}$ at the point $x_0$ if there exists a
parametrized family of points $\gamma:\R_{\ge 0}\to \R^{n}$ such that,
for every $t\in \R_{\ge 0}$, we have $\gamma(t)\in
\mathrm{R}_{\mathcal{X}}(\le t,x_0)$ and 
\begin{equation}\label{eq:2}
\gamma(t)=x_0+t^k\mathbf{v}+\mathcal{O}(t^{k+1}). 
\end{equation}
The set of all variations of $k$th order for the system $\mathcal{X}$
at the point $x_0$ is denoted by $\mathcal{K}^k_{\mathcal{X}}(x_0)$. We also
define the cone $\widehat{\mathcal{K}}^k_{\mathcal{X}}(x_0)$ by
\begin{align*}
\widehat{\mathcal{K}}^k_{\mathcal{X}}(x_0)=\bigcup_{\alpha\ge
  0}\alpha \left(\mathcal{K}^k_{\mathcal{X}}(x_0)\right).
\end{align*}
\end{definition}
Note that, in Definition~\ref{def:control_variations},
  there is no restriction on the regularity of the parametrized family
  of points $\gamma$. The next theorem shows how these control variations can be used to deduce the small-time local
controllability and the growth rate conditions for control-affine systems.


\begin{theorem}[\textbf{Characterization of growth rate condition}]\label{lem:characterization_growth_rate}
Let $(\mathcal{X},\mathfrak{C})$ be a $\mathrm{C}^{\nu}$ control-affine
system on $\R^n$ which satisfies Assumption~\ref{asump}, $x_0\in
\R^n$ be an equilibrium point for $\mathcal{X}$, and $N\in \N$. Then the following statements are equivalent:
\begin{enumerate}
\item[(i)] the control-affine system $\mathcal{X}$ satisfies the growth rate
  condition of order $N$;
\item[(ii)] $\widehat{\mathcal{K}}^N_{\mathcal{X}}(x_0)=\R^n$.
\end{enumerate}
\end{theorem}
\begin{proof}
$(i)$$\Rightarrow$$(ii)$: Since the control-affine system $(\mathcal{X},\mathfrak{C})$
  satisfies the growth rate condition of order $N$, there exists $T>0$ and $C>0$ such
  that, for every $t\in [0,T]$, we have $\overline{\mathrm{B}}(x_0,Ct^N)\subseteq
\mathrm{R}_{\mathcal{X}}(\le t,x_0)$. For every $\mathbf{v}\in
\mathbb{S}^{n-1}$, we define the family of points $\gamma_{\mathbf{v}}:[0,T]\to \overline{\mathrm{B}}(x_0,Ct^N)$ by
\begin{equation*}
\gamma_{\mathbf{v}}(t)=x_0+Ct^N\mathbf{v}.
\end{equation*}
By Definition~\ref{def:control_variations}, the vector
$\mathbf{v}$ is control variations of order $N$ for the system $\mathcal{X}$ at
the point $x_0$. This means that we have $\widehat{\mathcal{K}}^N_{\mathcal{X}}(x_0)=\R^n$ and therefore $(ii)$ holds.

$(ii)$$\Rightarrow$ $(i)$: The proof of this part follows from~\cite[Corollary 2.5]{MK:1990}.
\end{proof}
Roughly speaking, Theorem~\ref{lem:characterization_growth_rate} states that small-time local
controllability of a system is checkable using variations of order
$N$ if and only if the system satisfies the growth rate condition of order
$N$. Note that this result holds for control-affine systems in any
regularity class $\mathrm{C}^{\nu}$.

\section{Polynomial perturbations of real analytic control-affine systems}\label{sec:perturbations}
In this section, we focus on the class of real analytic control-affine systems
which satisfy the growth rate condition of order $N$. For a control
system in this class, we construct a multi-valued mapping to study the effect of perturbation of
vector fields of the system on its reachable sets. Let $(\mathcal{X}=\{X_0,X_1,\ldots,X_m\},\mathfrak{C})$ be a
real analytic control-affine system which satisfies the growth rate condition of
order $N$ at the point $x_0$ and let $(\mathcal{Y}=\{Y_0,Y_1,\ldots,Y_m\},\mathfrak{C})$ be another real
analytic control-affine system which we consider as the perturbed control
system. Moreover, we assume that both control-affine
  systems $(\mathcal{X},\mathfrak{C})$ and $(\mathcal{Y},\mathfrak{C})$ satisfy Assumption~\ref{asump}. Since
$(\mathcal{X},\mathfrak{C})$ satisfies the growth rate condition of order $N$, there exists $C,T>0$ such that 
\begin{align*}
\overline{\mathrm{B}}(x_0,Ct^N)\subseteq \mathrm{R}_{\mathcal{X}}(\le t,x_0),\qquad\forall t\le T.
\end{align*}
For every $t\in [0,T]$, we define a perturbation mapping $F^t_{\mathcal{X},\mathcal{Y}}:\overline{\mathrm{B}}(x_0,\tfrac{C}{2}t^N)
\rightrightarrows \mathrm{R}_{\mathcal{Y}}(\le t,x_0)$ which captures the
transition from reachable sets of $(\mathcal{X},\mathfrak{C})$ to the reachable sets
of $(\mathcal{Y},\mathfrak{C})$. The multi-valued mapping
$F^t_{\mathcal{X},\mathcal{Y}}$ is defined as composition of two mappings
$\xi^t_{\mathcal{Y}}$ and $\eta^t_{\mathcal{X}}$ as follows:
\begin{equation*}
F^t_{\mathcal{X},\mathcal{Y}}(x)=\xi^t_{\mathcal{Y}}\scirc
\eta^t_{\mathcal{X}}(x),\qquad\forall x\in \overline{\mathrm{B}}(x_0,\tfrac{C}{2}t^N).
\end{equation*}
The idea is that the mapping $\eta^t_{\mathcal{X}}$ takes a point in the closed
ball $\overline{\mathrm{B}}(x_0,\tfrac{C}{2}t^N)$ and gives the
switching times associated to that point and the mapping
$\xi^t_{\mathcal{Y}}$ takes a set of switching times and gives the
associated point in the reachable sets of the control system $(\mathcal{Y},\mathfrak{C})$.

We start by rigorously constructing the multi-valued mapping
$\eta^{t}_{\mathcal{X}}$. Since $(\mathcal{X},\mathfrak{C})$ satisfies
Assumption~\ref{asump} and the growth rate condition of order $N$, we have
$\mathfrak{C}=[-1,1]^m$ and $\overline{\mathrm{B}}(x_0,Ct^N)\subseteq
\mathrm{R}_{\mathcal{X}}(\le t,x_0)$, for every $t\le T$. This implies
that, for every $t\le T$, we get 
\begin{align*}
\overline{\mathrm{B}}(x_0,\tfrac{C}{2}t^N)\subset
  \mathrm{B}(x_0,Ct^N)\subseteq
  \mathrm{int}(\mathrm{R}_{\mathcal{X}}(\le t,x_0))
\end{align*}
By Theorem~\ref{th:1}, for every $x\in \overline{\mathrm{B}}(x_0,\tfrac{C}{2}t^N)$, there exist $p_x\in \N$,
$s_1,s_2,\ldots,s_{p_x}\in \R_{>0}$, and $\mathbf{w}^1,\mathbf{w}^2,\ldots,\mathbf{w}^{p_x}\in
[-1,1]^m$ such that $s_1+s_2+\ldots+s_{p_x}<t$ and
\begin{equation*}
\exp(s_1X_{\mathbf{w}^1})\scirc \exp(s_2X_{\mathbf{w}^2})\scirc\ldots\scirc \exp(s_{p_x} X_{\mathbf{w}^{p_x}})(x_0)=x.
\end{equation*}
Moreover, there exists an open neighborhood $V$ of $(s_1,s_2,\ldots,s_{p_x})$ in
$\R^{p_x}_{>0}$ such that the map $\xi^x_{\mathcal{X}}:V\to \R^n$, defined by
\begin{equation*}
\xi^x_{\mathcal{X}}(t_1,t_2,\ldots,t_{p_x})=\exp(t_1X_{\mathbf{w}^1})\scirc \exp(t_2X_{\mathbf{w}^2})\scirc\ldots\scirc \exp(t_{p_x} X_{\mathbf{w}^{p_x}})(x_0)
\end{equation*}
is $\mathrm{C}^ 1$ and of rank $n$ on $V$. Without loss of generality we can assume that,
for every $(t_1,t_2,\ldots,t_{p_x})\in V$, we have
\begin{equation*}
t_1+t_2+\ldots+t_{p_x}\le t.
\end{equation*}
By \cite[Lemma 2.2]{IK-PWM-JS:1993}, there exists a submanifold $M_x$
of $V$ containing $(s_1,s_2,\ldots,s_{p_x})$ such that $\xi^x_{\mathcal{X}}(M_x)$ is an open neighborhood of
$x$ in $\R^n$ and $\xi^x_{\mathcal{X}}\mid_{M_x}$ is a
$\mathrm{C}^{1}$-diffeomorphism. Let $S_x$ be an open neighborhood of
$(s_1,s_2,\ldots,s_{p_x})$ in $M_x$ such that $\overline{S}_{x}\subseteq
M_x$. Since $\overline{\mathrm{B}}(x_0,\tfrac{C}{2}t^N)$ is compact and, for
every $x\in \overline{\mathrm{B}}(x_0,\tfrac{C}{2}t^N)$, the set $\xi^{x}_{\mathcal{X}}(S_{x})$ is open in $\R^n$, there exists
$x_1,x_2,\ldots,x_r\in \overline{\mathrm{B}}(x_0,\tfrac{C}{2}t^N)$ such that 
\begin{equation*}
\overline{\mathrm{B}}(x_0,\tfrac{C}{2}t^N)\subseteq \bigcup_{i=1}^{r} \xi^{x_i}_{\mathcal{X}}(S_{x_i}).
\end{equation*} 
Now let us define $p=p_{x_1}+p_{x_2}+\ldots+p_{x_r}$ and let
$(\mathbf{u}^1,\mathbf{u}^2,\ldots,\mathbf{u}^p)$ be the ordered set obtained by concatenation
of the controls $(\mathbf{w}^1,\mathbf{w}^2,\ldots,\mathbf{w}^{p_{x_k}})$ for $1\le k\le r$. Since,
$\R^{p_{x_i}}\subseteq \R^{p}$, for every $i\in\{1,2,\ldots,r\}$, one can consider $S_{x_i}$ as a submanifold of $\R^p$. We define the multi-valued map
$\eta^t_{\mathcal{X}}:\overline{\mathrm{B}}(x_0,\tfrac{C}{2}t^N)\rightrightarrows \bigcup_{i=1}^{r} \overline{S}_{x_i}$ as
\begin{equation*}
\eta^t_{\mathcal{X}}(x)=\bigcup_{i\in
  \{1,2,\ldots,r\}}\left\{\left(\xi^{x_i}_{\mathcal{X}}\right)^{-1}(x)\suchthat
x\in \xi^{x_i}_{\mathcal{X}}(\overline{S}_{x_i})\right\},\qquad\forall
x\in \overline{\mathrm{B}}(x_0,\tfrac{C}{2}t^N).
\end{equation*}
Note that, for every $x\in \overline{\mathrm{B}}(x_0,\tfrac{C}{2}t^N)$, the number of
elements in $\eta^t_{\mathcal{X}}(x)$ is at most $r$.

The next step is to construct the single-valued mapping
$\xi^t_{\mathcal{Y}}:\R^p\to \mathrm{R}_{\mathcal{Y}}(\le t,x_0)$. We
define the map $\xi^t_{\mathcal{Y}}:\bigcup_{i=1}^{r}
\overline{S}_{x_i} \to \R^n$ by
\begin{equation*}
\xi^t_{\mathcal{Y}}(t_1,t_2,\ldots,t_p)=\exp(t_1Y_{\mathbf{u}^1})\scirc \exp(t_2Y_{\mathbf{u}^2})\scirc\ldots\scirc \exp(t_{p} Y_{\mathbf{u}^p})(x_0).
\end{equation*}
Then the multi-valued map
$F^t_{\mathcal{X},\mathcal{Y}}:\overline{\mathrm{B}}(x_0,\tfrac{C}{2}t^N)\rightrightarrows
\R^n$ is given by
\begin{equation*}
F^t_{\mathcal{X},\mathcal{Y}}(x)=\xi^t_{\mathcal{Y}}\scirc \eta^t_{\mathcal{X}}(x).
\end{equation*}
One can observe that the mapping $F^t_{\mathcal{X},\mathcal{Y}}$ is
finite-valued and has the following regularity properties.

\begin{theorem}[\textbf{Regularity of the perturbation mapping}]\label{lem:multi-valued_map}
Let $(\mathcal{X} = \{X_0,\ldots,X_m\},\mathfrak{C})$ and
$(\mathcal{Y}=\{Y_0,\ldots,Y_m\},\mathfrak{C})$ be two real analytic
control-affine system which satisfy Assumption~\ref{asump}. Suppose that, for $x_0\in \mathbb{R}^n$, there exists $N\in \N$ such that 
\begin{enumerate}[label=(\alph*)]
\item the system $(\mathcal{X},\mathfrak{C})$ satisfies
  the growth rate condition of order $N$,
\item for every $i\in \{0,\ldots,m\}$, the vector fields $X_i$ and $Y_i$
have the same Taylor polynomial of order $N$ around $x_0$.
\end{enumerate}
For every $t\in [0,T]$, let the map $F^t_{\mathcal{X},\mathcal{Y}}:\overline{\mathrm{B}}(x_0,\tfrac{C}{2}t^N)\to
\mathrm{R}_{\mathcal{Y}}(\le t, x_0)$ be defined as above. Then the following statements hold:
\begin{enumerate}[label=(\roman*)]
\item\label{p1:reg_spatial}  For every $t\in [0,T]$ and every $x\in
  \overline{\mathrm{B}}(x_0,\tfrac{C}{2}t^N)$, there exist a positive integer $l\in \N$, a neighborhood $W$ containing $x$, and
continuous functions $f^1,f^2,\ldots,f^l:W\to \R^n$ such that
\begin{equation*}
\left\{f^1(y)\right\}\subseteq F^t_{\mathcal{X},\mathcal{Y}}(y)\subseteq\left\{f^1(y), f^2(y), \ldots,
  f^l(y)\right\},\qquad\forall y\in W.
\end{equation*}
\item\label{p2:reg_time} there exist $\alpha>0$ and $T_{\min}\in (0,T)$ such that, for every 
  $t\le T_{\min}$ and every $x\in \overline{\mathrm{B}}(x_0,\tfrac{C}{2}t^N)$, we have
\begin{equation*}
\|y-x\|\le \alpha t^{N+1},\qquad \forall y\in F^t_{\mathcal{X},\mathcal{Y}}(x).
\end{equation*}
\end{enumerate}
\end{theorem}

\begin{proof}
Regarding part~\ref{p1:reg_spatial}, suppose that, for every $i\in
\{1,2,\ldots, r\}$, the map
$\xi^{x_i}_{\mathcal{X}}$ and the manifold $S_{x_i}$ are defined as
above. Since $\overline{\mathrm{B}}(x_0,\tfrac{C}{2}t^N)\subseteq
\bigcup_{i=1}^{r} \xi^{x_i}_{\mathcal{X}}(S_{x_i})$, we have $x\in \bigcup_{i=1}^{r} \xi^{x_i}_{\mathcal{X}}(S_{x_i})$. Without loss of
generality, we can assume that
\begin{equation*}
x\in \xi^{x_1}_{\mathcal{X}}(S_{x_1}).
\end{equation*}
Note that $\xi^{x_1}_{\mathcal{X}}(S_{x_1})$ is open in $\R^n$. Therefore, there exists a
neighborhood $U$ of $x$ such that $U\subseteq
\xi^{x_1}_{\mathcal{X}}(S_{x_1})$. On the other hand, since we have $\overline{\mathrm{B}}(x_0,\tfrac{C}{2}t^N)\subseteq
\bigcup_{i=1}^{r} \xi^{x_i}_{\mathcal{X}}(\overline{S}_{x_i})$, without loss of generality, we
can assume that there exists $l\in \{1,2,\ldots,r\}$ such that
\begin{eqnarray*}
x\in \xi^{x_i}_{\mathcal{X}}(\overline{S}_{x_i}), &\quad i\in\{1,2,\ldots,l\},\\
x\not\in \xi^{x_i}_{\mathcal{X}}(\overline{S}_{x_i}), &\quad  i\in\{l+1, l+2,\ldots,r\}.
\end{eqnarray*}
For every $i\in\{l+1,l+2,\ldots,r\}$, the set
$\xi^{x_i}_{\mathcal{X}}(\overline{S}_{x_i})$ is closed in $\R^n$. Therefore,
\begin{equation*}
\bigcup_{i=l+1}^{r} \xi^{x_i}_{\mathcal{X}}(\overline{S}_{x_i})
\end{equation*}
is closed in $\R^n$. Moreover, we know that $x\not\in \bigcup_{i=l+1}^{s}
\xi^{x_i}_{\mathcal{X}}(\overline{S}_{x_i})$. This implies that there exists a
neighborhood $V$ of $x$ such that
\begin{equation*}
V\cap \left(\bigcup_{i=l+1}^{r} \xi^{x_i}_{\mathcal{X}}(\overline{S}_{x_i})\right)=\emptyset.
\end{equation*}
Note that, for every $i\in \{1,2,\ldots,l\}$,
$\overline{S}_{x_i}\subseteq M_{x_i}$. Therefore, for every $i\in \{1,2,\ldots,l\}$, we have
$\xi^{x_i}_{\mathcal{X}}(\overline{S}_{x_i})\subseteq
\xi^{x_i}_{\mathcal{X}}(M_{x_i})$. Since $x\in
\xi^{x_i}_{\mathcal{X}}(M_{x_i})$, for every $i\in \{1,2,\ldots,l\}$, the
set $\bigcap_{i=1}^{l}\xi^{x_i}_{\mathcal{X}}(M_{x_i})$ is nonempty. We set 
\begin{equation*}
W=\left(\bigcap_{i=1}^{l}\xi^{x_i}_{\mathcal{X}}(M_{x_i})\right)\cap V\cap U.
\end{equation*} 
For every $i\in \{1,2,\ldots,l\}$, we define the function $f^i:W\to
\R^n$ as 
\begin{equation*}
f^i(y)=\xi^t_{\mathcal{Y}}\scirc \left(\xi^{x_i}_{\mathcal{X}}\right)^{-1}(y),
\qquad\forall y\in W.
\end{equation*}
Note that, for every $i\in \{1,2,\ldots,l\}$, the map $\xi^{x_i}_{\mathcal{X}}$
is a $\mathrm{C}^{1}$-diffeomorphism on $M_{x_i}$. Therefore, for every $i\in
\{1,2,\ldots,l\}$, the map $f^i:W\to \R^n$ is continuous. Now, it is
clear from the definition of $F^t_{\mathcal{X},\mathcal{Y}}$ that we have
\begin{equation*}
F^t_{\mathcal{X},\mathcal{Y}}(t,x)=\left\{f^1(x), f^2(x), \ldots, f^l(x)\right\}.
\end{equation*}
Since $W\subseteq V$ and $V$ is chosen such that $V\cap
\left(\bigcup_{i=l+1}^{r}
  \xi^{x_i}_{\mathcal{X}}(\overline{S}_{x_i})\right)=\emptyset$, for every $i\in\{l+1,l+2,\ldots,r\}$ and every
$y\in W$, we have
\begin{equation*}
\xi^t_{\mathcal{Y}}\scirc \left(\xi^{x_i}_{\mathcal{X}}\right)^{-1}(y)\not\in
F^t_{\mathcal{X}, \mathcal{Y}}(y). 
\end{equation*}
Thus, we have 
\begin{equation*}
F^t_{\mathcal{X},\mathcal{Y}}(y)\subseteq\left\{f^1(y), f^2(y), \ldots,
  f^l(y)\right\},\qquad\forall y\in W.
\end{equation*}
Finally, since $W\subseteq U$, and $U$ is chosen such that $U\subseteq \xi^{x_1}_{\mathcal{X}}(S_{x_1})$, 
for every $y\in W$, 
\begin{equation*}
\xi^t_{\mathcal{Y}}\scirc \left(\xi^{x_1}_{\mathcal{X}}\right)^{-1}(y)\in
F^t_{\mathcal{X}, \mathcal{Y}}(y). 
\end{equation*}
Therefore, for every $y\in W$, we have $f^{1}(y)\in F^t_{\mathcal{X},
  \mathcal{Y}}(y)$. 

Regarding part~\ref{p2:reg_time}, note that the system
$(\mathcal{X},\mathfrak{C})$ satisfies Assumption~\ref{asump} and
therefore, we have $\mathfrak{C}=[-1,1]^m$. We define the set $B\subseteq
\Gamma^{\omega}(\R^n)$ by
\begin{align*}
B = \left\{X_{\mathbf{u}}\mid \mathbf{u}\in [-1,1]^m\right\}\bigcup \left\{Y_{\mathbf{u}}\mid \mathbf{u}\in [-1,1]^m\right\}.
\end{align*}
Note that, for every $\mathbf{u}\in [-1,1]^m$, every compact set $K$,
every $C^{\omega}$-function $f$, and every $\mathbf{a}\in c^{\downarrow}_0$, we have 
\begin{align*}
\rho^{\omega}_{K,\mathbf{a},f}(X_{\mathbf{u}})=\rho^{\omega}_{K,\mathbf{a},f}(X_0+\sum_{i=1}^{m}u_iX_i)\le
  (m+1)\max\{\rho^{\omega}_{K,\mathbf{a},f}(X_i)\mid i\in \{0,1,\ldots,m\}\}.
\end{align*}
Similarly, for every $\mathbf{u}\in [-1,1]^m$, every compact set $K$,
every $C^{\omega}$-function $f$, and every $\mathbf{a}\in c^{\downarrow}_0$, we have 
\begin{align*}
\rho^{\omega}_{K,\mathbf{a},f}(Y_{\mathbf{u}})=\rho^{\omega}_{K,\mathbf{a},f}(Y_0+\sum_{i=1}^{m}u_iY_i)\le
  (m+1)\max\{\rho^{\omega}_{K,\mathbf{a},f}(Y_i)\mid i\in \{0,1,\ldots,m\}\}.
\end{align*}
For every $i\in \{0,1,\ldots,m\}$ we define
$L_i=\max\{\rho^{\omega}_{K,\mathbf{a},f}(X_i),\rho^{\omega}_{\mathbf{a},K,f}(Y_i)\}$. Thus,
if we define the constant $L\in \R_{>0}$ by
\begin{align*}
L=(m+1)\max\{L_i\mid i\in \{0,1,\ldots,m\}\},
\end{align*}
then we have 
\begin{align*}
\rho^{\omega}_{K,\mathbf{a},f}(v)\le L,\qquad\forall v\in B.
\end{align*}
By Theorem~\ref{thm:boundedset_locallyconvex}, this implies that the set $B$ is bounded in $\Gamma^{\omega}(\R^n)$. Using Theorem
\ref{cor:estimate_picard_iterations}, there exist $M,L>0$ and $\overline{T}<T_B$ such that, for every $t\in (0,\overline{T}]$ and every real analytic
vector field $Z\in\mathrm{PC}([0,\overline{T}] ;B)$, we have
\begin{equation*}
\left\|\mathrm{ev}_{x_0}\scirc\exp(tZ)(x^i)-\mathrm{ev}_{x_0}\scirc\exp_{N}(tZ)(x^i)\right\|\le
\frac{\left(Mt\right)^{N+1}}{1-Mt} L,
\end{equation*}
We set $\alpha=\sqrt{n}M^{N+1}L$ and $T_{\min}=\min\{\overline{T},T\}$
Let $t\in [0,T_{\min}]$ and $x\in
\overline{\mathrm{B}}(x_0,\tfrac{C}{2}t^N)$. If $y\in F^t_{\mathcal{X},\mathcal{Y}}(x)$, then there exist
$\mathbf{I}=(\mathbf{u}^1,\mathbf{u}^2,\ldots,\mathbf{u}^p)\in ([-1,1]^m)^p$
and $\mathbf{t}=(t_1,t_2,\ldots,t_p)\in \R^p_{\ge 0}$ such that
$|\mathbf{t}|< t$ and
\begin{eqnarray*}
 x&=&\exp(t_1X_{\mathbf{u}^1})\scirc \exp(t_2X_{\mathbf{u}^2})\scirc \ldots \scirc \exp(t_pX_{\mathbf{u}^p})(x_0)=\exp(|\mathbf{t}|X^{\mathbf{I},\mathbf{t}}),\\
y&=&\exp(t_1Y_{\mathbf{u}^1})\scirc \exp(t_2Y_{\mathbf{u}^2})\scirc \ldots \scirc \exp(t_pY_{\mathbf{u}^p})(x_0)=\exp(|\mathbf{t}|Y^{\mathbf{I},\mathbf{t}}).
\end{eqnarray*}
Note that, for every $\mathbf{k}\in \mathbb{Z}^n_{\ge 0}$ with the
property that $|\mathbf{k}|\le N$, we have
\begin{equation*}
D^{\mathbf{k}} X_i(x_0)=D^{\mathbf{k}} Y_i(x_0),\qquad i\in \{1,\ldots,m\}.
\end{equation*}
Note that the $N$th iteration in the sequence of flow iteration~\eqref{eq:Picards_iterations}
  only depends on the derivatives up to order $N$ of the vector field $X$. This implies that,
\begin{equation*}
\mathrm{ev}_{x_0}\scirc\exp_N(|\mathbf{t}|X^{\mathbf{I},\mathbf{t}})(x^i)=\mathrm{ev}_{x_0}\scirc\exp_N(|\mathbf{t}|Y^{\mathbf{I},\mathbf{t}})(x^i).
\end{equation*}
Thus we have 
\begin{multline*}
\left\|\mathrm{ev}_{x_0}\scirc\exp(|\mathbf{t}|X^{\mathbf{I},\mathbf{t}})(x^i)-\mathrm{ev}_{x_0}\scirc\exp(|\mathbf{t}|Y^{\mathbf{I},\mathbf{t}})(x^i)\right\|\\\le
\left\|\mathrm{ev}_{x_0}\scirc\exp(|\mathbf{t}|X^{\mathbf{I},\mathbf{t}})(x^i)-\mathrm{ev}_{x_0}\scirc\exp_N(|\mathbf{t}|X^{\mathbf{I},\mathbf{t}})(x^i)\right\|\\+\left\|\mathrm{ev}_{x_0}\scirc\exp(|\mathbf{t}|Y^{\mathbf{I},\mathbf{t}})(x^i)-\mathrm{ev}_{x_0}\scirc\exp_N(|\mathbf{t}|Y^{\mathbf{I},\mathbf{t}})(x^i)\right\|.
\end{multline*}
Since $X^{\mathbf{I},\mathbf{t}}, Y^{\mathbf{I},\mathbf{t}}\in
\mathrm{PC}([0,|\mathbf{t}|]; B)$, we have
\begin{eqnarray*}
\left\|\mathrm{ev}_{x_0}\scirc\exp(|\mathbf{t}|X^{\mathbf{I},\mathbf{t}})(x^i)-\mathrm{ev}_{x_0}\scirc\exp_N(|\mathbf{t}|X^{\mathbf{I},\mathbf{t}})(x^i)\right\|\le
  \frac{\left(Mt\right)^{N+1}}{1-Mt} L,\\
\left\|\mathrm{ev}_{x_0}\scirc\exp(|\mathbf{t}|Y^{\mathbf{I},\mathbf{t}})(x^i)-\mathrm{ev}_{x_0}\scirc\exp_{N}(|\mathbf{t}|Y^{\mathbf{I},\mathbf{t}})(x^i)\right\|
\le \frac{\left(Mt\right)^{N+1}}{1-Mt} L.
\end{eqnarray*}
Thus, we have
\begin{multline*}
\left\|\mathrm{ev}_{x_0}\scirc\exp(|\mathbf{t}|X^{\mathbf{I},\mathbf{t}})(x^i)-\mathrm{ev}_{x_0}\scirc\exp(|\mathbf{t}|Y^{\mathbf{I},\mathbf{t}})(x^i)\right\|\\\le
2\frac{\left(Mt\right)^{N+1}}{1-Mt} L,\qquad\forall i\in \{1,2,\ldots,n\}.
\end{multline*}
Therefore, we have
\begin{equation*}
\left\|\exp(|\mathbf{t}|X^{\mathbf{I},\mathbf{t}})(x_0)-\exp(|\mathbf{t}|Y^{\mathbf{I},\mathbf{t}})(x_0)\right\|\le
2\sqrt{n}\frac{\left(Mt\right)^{N+1}}{1-Mt} L.
\end{equation*}
Note that, by our choice of $t_1,t_2,\ldots,t_p\in \R_{>0}$, we have
\begin{eqnarray*}
y&=&\exp(|\mathbf{t}|Y^{\mathbf{I},\mathbf{t}})(x_0),\\
x&=&\exp(|\mathbf{t}|X^{\mathbf{I},\mathbf{t}})(x_0).
\end{eqnarray*}
Note that, by the definition, we have $\alpha=\sqrt{n}M^{N+1}L$ and
$Mt\le M\overline{T}\le\frac{1}{2}$. This implies that 
\begin{equation*}
\left\|y-x\right\|\le \alpha t^{N+1}.
\end{equation*}
Since $\alpha$ does not depend on $t$, the above inequality holds for
every $t\in [0,T_{\min}]$.
\end{proof}

\begin{rem}
Some remarks are in order. 
\begin{enumerate}[label=(\roman*)]
\item Part~\ref{p1:reg_spatial} of
Theorem~\ref{lem:multi-valued_map} can be considered as a 
regularity result for the finite-valued mapping $x\rightrightarrows
F^t_{\mathcal{X},\mathcal{Y}}(x)$ and part~\ref{p2:reg_time}
of Theorem~\ref{lem:multi-valued_map} can be considered as a regularity result for the
finite-valued mapping $t\rightrightarrows
F^t_{\mathcal{X},\mathcal{Y}}(x)$; 
\item For the proof of part~\ref{p1:reg_spatial} of Theorem~\ref{lem:multi-valued_map}, it is
not required to assume that the family of vector fields $\mathcal{X}$ and $\mathcal{Y}$ have
the same Taylor polynomials of order $N$ around the point $x_0$. In
fact, Theorem~\ref{lem:multi-valued_map}\ref{p1:reg_spatial} is true
for any arbitrary perturbation of the real analytic control
system $(\mathcal{X},\mathfrak{C})$ satisfying Assumption~\ref{asump}. However, this condition is essential for the
proof of Theorem~\ref{lem:multi-valued_map}\ref{p2:reg_time}.
\end{enumerate}
\end{rem}

\section{The main theorem}\label{sec:main_result}

In this section we prove the main result of this paper, which can be
considered as a robustness of the growth rate condition of
order $N$ with respect to polynominal perturbations of order higher than $N$. Roughly
speaking our main result states that, given a real analytic system $(\mathcal{X},\mathfrak{C})$ which
satisfies the growth rate condition of order $N$ at the point $x_0$, if we
perturb the vector fields of $(\mathcal{X},\mathfrak{C})$ around $x_0$ by Taylor polynomials
of order higher than $N$, then the resulting system again satisfies the growth rate condition of order $N$. 

\begin{theorem}[\textbf{Polynomial perturbations of real analytic systems}]\label{thm:main_result}
Let $(\mathcal{X}=\{X_0,\ldots, X_m\},\mathfrak{C})$ and
$(\mathcal{Y}=\{Y_0,\ldots,Y_m\}, \mathfrak{C})$ be two real analytic
control-affine systems on $\R^n$ satisfying
  Assumption~\ref{asump} and let $x_0\in \R^n$. Suppose that the following conditions hold:
\begin{enumerate}
\item[(i)] $x_0\in \mathbb{R}^n$ is an equilibrium point for control-affine systems $\mathcal{X}$ and $\mathcal{Y}$,
\item[(ii)] the control-affine system $\mathcal{X}$ satisfies the growth rate condition of order
  $N$ at the point $x_0$, and
\item[(iii)] for every $i\in \{0,\ldots,m\}$, the vector fields $X_i$
  and $Y_i$ have the same Taylor polynomial of order $N$ around $x_0$.
\end{enumerate} 
Then $\mathcal{Y}$ satisfies the growth rate condition of order $N$ at
the point $x_0$. 
\end{theorem}
\begin{proof}
Since $(\mathcal{X},\mathfrak{C})$ satisfies the growth rate condition of order $N$
at the point $x_0$, there
exist $T,C>0$ such that, for every $t\in [0,T]$, we have 
\begin{equation*}
\overline{\mathrm{B}}(x_0,Ct^N)\in \mathrm{R}_{\mathcal{X}}(\le t,x_0).
\end{equation*}
For every $\mathbf{v}\in \mathbb{S}^n$, consider the
parametrized family of points $\gamma_{\mathbf{v}}:[0,T]\to \R^n$ defined by
\begin{equation*}
\gamma_{\mathbf{v}}(t)=x_0+\tfrac{C}{2}t^N\mathbf{v}.
\end{equation*}
It is clear that, for every $t\in [0,T]$, we have
$\gamma_{\mathbf{v}}(t)\in
\overline{\mathrm{B}}(x_0,\tfrac{C}{2}t^N)$. For every
$t\in [0,T]$, recall the definition of the multi-valued map
$F^t_{\mathcal{X},\mathcal{Y}}$ in Section~\ref{sec:perturbations} and let
$f^t_{\mathcal{X},\mathcal{Y}}:\overline{\mathrm{B}}(x_0,\tfrac{C}{2}t^N)\to
\mathrm{R}_{\mathcal{Y}}(\le t,x_0)$ be a single-valued selection of the multi-valued mapping $F^t_{\mathcal{X},\mathcal{Y}}$. Then, for every
$\mathbf{v}\in \mathbb{S}^{n-1}$, we define the curves $\mu_{\mathbf{v}}:[0,T]\to \R^n$ by 
\begin{equation*}
\mu_{\mathbf{v}}(t)=f^t_{\mathcal{X},\mathcal{Y}}(\gamma_{\mathbf{v}}(t)).
\end{equation*}
Note that, for every $t\in [0,T]$, we have $\gamma_{\mathbf{v}}(t)\in
\overline{\mathrm{B}}(x_0,\tfrac{C}{2}t^N)$. Therefore, for every $t\in [0,T]$, the map $\mu_{\mathbf{v}}(t)$ is
well-defined and $\mu_{\mathbf{v}}(t)\in \mathrm{R}_{\mathcal{Y}}(\le t,x_0)$. 

Now, by Theorem~\ref{lem:multi-valued_map} part~\ref{p2:reg_time}, there
exist $\alpha>0$ and $0<T_{\min}<T$ such that, for every $t\le
T_{\min}$ and every $x\in \overline{\mathrm{B}}(x_0,\tfrac{C}{2}t^N)$,
we get
\begin{align*}
\|f^t_{\mathcal{X},\mathcal{Y}}(\gamma_{\mathbf{v}}(t))-\gamma_{\mathbf{v}}(t)\|\le \alpha t^{N+1}.
\end{align*} 
This implies that, for every $t\in [0,T_{\min}]$, 
\begin{align*}
\mu_{\mathbf{v}}(t) = f^t_{\mathcal{X},\mathcal{Y}}(\gamma_{\mathbf{v}}(t))=\gamma_{\mathbf{v}}(t)+\mathcal{O}(t^{N+1})
  = x_0+\tfrac{C}{2}t^N\mathbf{v}+\mathcal{O}(t^{N+1}).
\end{align*} 
Thus, using Definition~\ref{def:control_variations}, $\mathbf{v}$ is a
control variation of $N$th order for the control-affine system
$(\mathcal{Y},\mathfrak{C})$. Thus, we have $\widehat{\mathcal{K}}^N_{\mathcal{Y}}=\bigcup_{\alpha\ge 0}
\alpha\mathcal{K}^N_{\mathcal{Y}}(x_0)=\R^n$ and, by Theorem~\ref{lem:characterization_growth_rate}, the control-affine system
$(\mathcal{Y},\mathfrak{C})$ satisfies the growth rate condition of order $N$.
\end{proof}

As a direct consequence of Theorem~\ref{thm:main_result}, one can
prove the following connection between Conjecture~\ref{conj:2} and Conjecture~\ref{conj:1}. 
\begin{corollary}\label{thm:connection_conjectures}
If Conjecture~\ref{conj:1} is true, then Conjecture~\ref{conj:2} holds.
\end{corollary}

\section{Conclusion}

In this paper, we studied small-time local controllability of real
analytic control-affine systems under polynomial perturbations of their vector
fields. For a real analytic control-affine system which satisfies the
growth rate condition of order $N$, we construct a suitable
multi-valued map for studying the perturbations of its reachable
sets. We showed that if a real analytic control-affine system
satisfies the growth rate condition of order $N$, then any
perturbation of the system by polynomial vector fields of order higher
than $N$ is again small-time local controllable. For the future research,
it is interesting to study suitable filterations of the vector fields of
the system (see for example~\cite{HH:91}) to sharpen the perturbation results obtained in
this paper. Moreover, it is worth mentioning that while our techniques heavily depend on the real
analyticity of the vector fields of the system, it remains to be seen
whether the main result of this paper still holds for $C^{\infty}$ control-affine systems.

\section*{Acknowledgments} The author would like to thank Professor
Andrew Lewis and the anonymous reviewers for reading this manuscript
carefully and for their constructive suggestions which improved the quality and exposition of the paper.

\appendix
\section{Proof of Theorem~\ref{cor:estimate_picard_iterations}}\label{app:1}

In this appendix, we present a proof of the Theorem~\ref{cor:estimate_picard_iterations} using the
$\mathrm{C}^{\omega}$-topology on the space
$\mathrm{LC}^{\omega}(\R^n)$. As mentioned in Section~\ref{sec:functions_vector_fields}, it is
also possible to give a proof of this theorem using the estimates in~\cite[\S 2.4.4]{agrachevcontrol2004} for
piecewise constant real analytic vector fields as curves on the space
$\mathrm{\Gamma}^{\infty}(\R^n)$ (see~\cite[\S 2.4.4 and Appendix A.2]{agrachevcontrol2004}). Note that, in~\cite[\S
2.4.4]{agrachevcontrol2004}, the space $\Gamma^{\infty}(\R^n)$ is equipped with the Whitney compact-open
topology~\cite[\S 2.2]{agrachevcontrol2004}. 
\begin{proof}
By Theorem~\ref{th:3}, for every $i\in \{1,2,\ldots,n\}$, every compact
set $K\subset \R^n$ containing $x_0$, and every $\mathbf{a}\in
c^{\downarrow}_{0}$, there exist $M, M_{x^i}>0$ such that
\begin{equation}\label{eq:estimate_picards}
\rho^{\omega}_{K,\mathbf{a},x^i}(\widehat{\exp}_{k}(t X)-\widehat{\exp}_{k-1}(t X))\le (Mt)^{k+1}M_{x^i}.
\end{equation} 
Note that, for every $\phi\in \mathrm{LC}^{\omega}(\R^n)$ and every
$i\in \{1,2,\ldots,n\}$, we have $\mathrm{ev}_{x_0}\scirc
\phi(x^i)=\phi(x^i)(x_0)$. Therefore, we get
\begin{equation*}
\|\mathrm{ev}_{x_0}\scirc\phi(x^i)\|\le \sup\{\phi(x^i)(y)\mid y\in
K\}\le \rho^{\omega}_{K,\mathbf{a},x^i}(\phi).
\end{equation*}
Now, by setting $\max_i\{M_{x^i}\}=L$ and using the estimate~\eqref{eq:estimate_picards}, we get 
\begin{equation*}
\|\mathrm{ev}_{x_0}\scirc\widehat{\exp}_{k+1}(t
X)(x^i)-\mathrm{ev}_{x_0}\scirc\widehat{\exp}_{k}(t X)(x^i)\|\le 
(Mt)^{k+1}L,\qquad\forall i\in\{1,2,\ldots,n\}.
\end{equation*}
Therefore, if we choose $\overline{T}\le T_{B}$ such that
$M\overline{T}<1$, we have 
\begin{align*}
\|\mathrm{ev}&_{x_0}\scirc\widehat{\exp}(t
X)(x^i)-\mathrm{ev}_{x_0}\scirc\widehat{\exp}_{k}(t X)(x^i)\|\\ &\le \sum_{r=k}^{\infty}\|\mathrm{ev}_{x_0}\scirc\widehat{\exp}_{r+1}(t
X)(x^i)-\mathrm{ev}_{x_0}\scirc\widehat{\exp}_{r}(t X)(x^i)\| \\ & \le
                                                                   \sum_{r=k}^{\infty}(Mt)^{r+1}L=\frac{(Mt)^{k+1}}{1-Mt}L,\qquad\forall
                                                                   t\in [0,\overline{T}].
\end{align*}
This completes the proof of the theorem.
\end{proof}

\bibliographystyle{plainurl}
\bibliography{Ref.bib}
\end{document}